\documentclass[a4paper,11pt]{amsart}
\usepackage{amsmath,amsfonts,amssymb,amsthm,latexsym,amscd}
\usepackage[dvips]{graphicx}
\usepackage{parskip}
\def\a{\alpha}
\def\C{\mathcal{C}}
\def\wC{\widetilde{C}}
\def\wD{\widetilde{D}}
\def\O{\Omega}
\def\c{\operatorname{cr}}
\def\comp{\operatorname{comp}}
\def\n{\nabla}

\def\o{\omega}
\def\s{\operatorname{sign}}
\def\sig{\sigma}

\def\B{\mathbf}

\newcommand\ris[5]{\raisebox{#1mm}{\hspace{#2mm}\includegraphics[width=#3mm]{#4.eps}\hspace{#5mm}}}

\newtheorem{thm}{Theorem}[section]
\newtheorem{thm*}{Theorem}

\newtheorem{prop}[thm]{Proposition}
\newtheorem{cor}[thm]{Corollary}
\newtheorem*{cor*}{Corollary}
\newtheorem{defn}[thm]{Definition}
\newtheorem{rem}[thm]{Remark}
\newtheorem{ex}[thm]{Example}

\newtheorem*{q*}{Question}

\begin{document}

\title[Conway-type polynomial of closed braids]{Coloring link diagrams and Conway-type polynomial of braids}

\author{Michael Brandenbursky}

\vspace{2cm}

\begin{abstract}
In this paper we define and present a
simple combinatorial formula for a 3-variable Laurent polynomial invariant $I(a,z,t)$ of conjugacy classes in Artin braid group $\B B_m$. We show that the Laurent polynomial $I(a,z,t)$ satisfies the Conway skein relation and the coefficients of the 1-variable polynomial $t^{-k}I(a,z,t)|_{a=1,t=0}$ are Vassiliev invariants of braids.
\end{abstract}

\maketitle

\section{Introduction.}

In this work we consider link invariants arising from the Alexander-Conway
and HOMFLY-PT polynomials. The HOMFLY-PT polynomial $P(L)$
is an invariant of an oriented link $L$ (see for example \cite{FYHLMO},
\cite{LM}, \cite{PT}). It is a Laurent polynomial in two
variables $a$ and $z$, which satisfies the following skein relation:
\begin{equation}\label{eq:Homfly-skein}
aP\left(\ris{-4}{-1}{10}{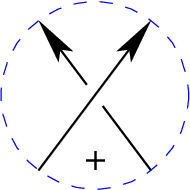}{-1.1}\right)-
a^{-1}P\left(\ris{-4}{-1}{10}{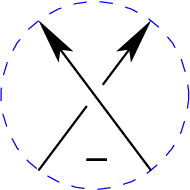}{-1.1}\right)=
zP\left(\ris{-4}{-1.1}{10}{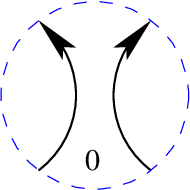}{-1.1}\right).
\end{equation}
The HOMFLY-PT polynomial is normalized in the following way.
If $O_r$ is the $r$-component unlink, then
$P(O_r)=\left(\frac{a-a^{-1}}{z}\right)^{r-1}$. The Conway
polynomial $\nabla$ may be defined as $\nabla(L):=P(L)|_{a=1}$.
This polynomial is a renormalized version of the Alexander
polynomial (see for example \cite{Conway}, \cite{Likorish}). All coefficients
of $\nabla$ are finite type or Vassiliev invariants.

Recently, invariants of conjugacy classes of braids received a considerable attention, since in some cases they define quasi-morphisms on braid groups and induce quasi-morphisms on certain groups of diffeomorphisms of smooth manifolds, see for example \cite{B,BK,Cal,CHH,w-signature,surfaces,HKM1,HKM2,Mal2,Mal1}.

In this paper we present a certain combinatorial construction of a 3-variable Laurent polynomial invariant $I(a,z,t)$ of conjugacy classes in Artin braid group $\B B_m$. We show that the polynomial $I(a,z,t)$ satisfies the Conway skein relation and the coefficients of the  polynomial $t^{-k}I(a,z,t)|_{a=1,t=0}$ are finite type invariants of braids for every $k\geq 2$.
We modify the polynomial $t^{-2}I(a,z,t)|_{a=1,t=0}$, so that the resulting polynomial is a polynomial invariant of links. In addition, we show that this polynomial equals to $zP'_a|_{a=1}$, where $P'_a|_{a=1}$ is the partial derivative of the HOMFLY-PT polynomial, w.r.t. the variable $a$, evaluated at $a=1$. Another interpretation of the later polynomial was recently given by the author in \cite{B1,B2}.

\subsection{Construction of the polynomial $I(a,z,t)$}
\label{subsec-construction}

Recall that the Artin braid group $\B B_m$ on $m$ strings has the following presentation:
\begin{equation}\label{eq:braid-presentation}
\B B_m=\langle\sig_1,\ldots,\sig_{m-1}|\hspace{2mm} \sig_i\sig_j=\sig_j\sig_i,\hspace{1mm}|i-j|\geq2;\hspace{2mm}\sig_i\sig_{i+1}\sig_i=\sig_{i+1}\sig_i\sig_{i+1}\rangle,
\end{equation}
where each generator $\sig_i$ is shown in Figure \ref{fig:braid-gen-sig-i}a.
\begin{figure}[htb]
\centerline{\includegraphics[height=1.6in]{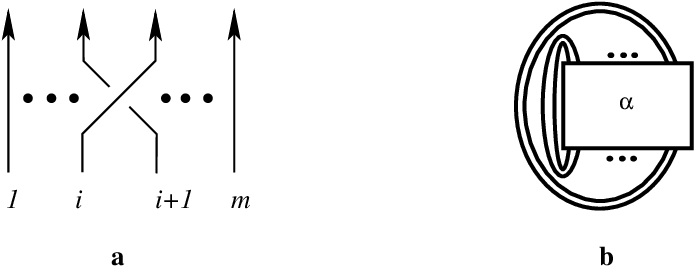}}
\caption{\label{fig:braid-gen-sig-i} Artin generator $\sig_i$ and a closure of a braid $\a$.}
\end{figure}
Let $\a\in \B B_m$. We take any representative of $\a$ and connect it opposite ends by simple nonintersecting curves as shown in Figure \ref{fig:braid-gen-sig-i}b and obtain the oriented link diagram $D$. We impose an equivalence relation on the set such diagrams as follows. Two such diagrams are equivalent if one can pass from one to another by a finite sequence of $\O2a$, $\O2b$ and $\O3$ Reidemeister moves shown in Figure \ref{fig:Reid-moves}.
%%%%%%%%%%%%%%%%%%%%%%%%%%%%%%%%%%%%%%%%%%
\begin{figure}[htb]
\centerline{\includegraphics[height=3.3in]{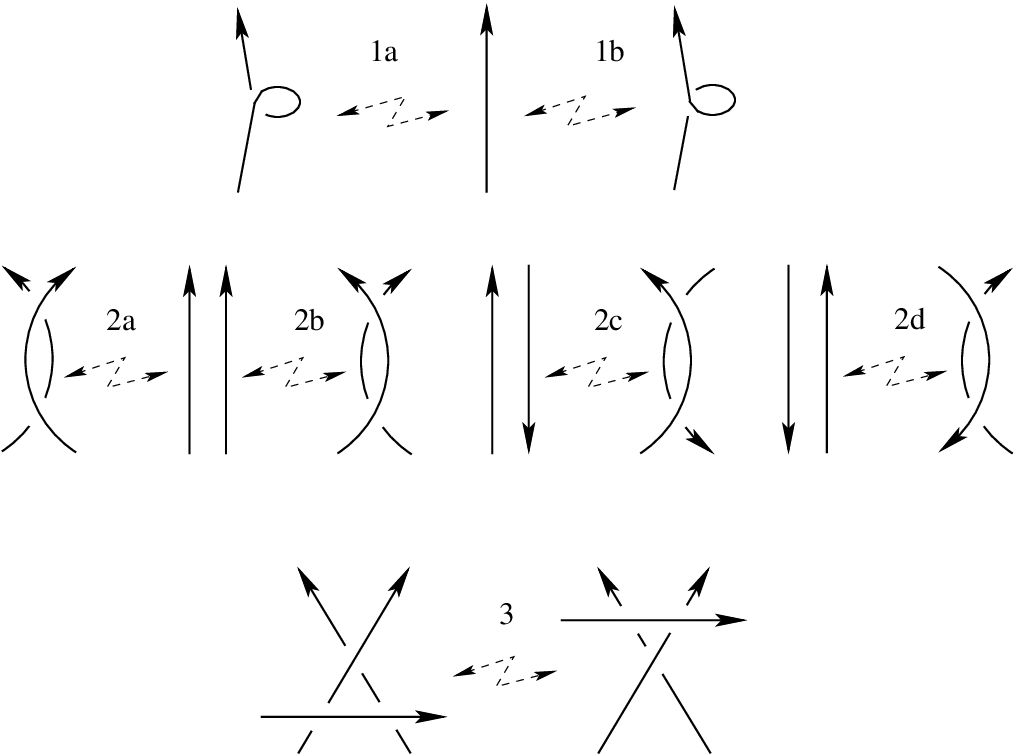}}
\caption{\label{fig:Reid-moves} $\O1a$, $\O1b$, $\O2a$, $\O2b$, $\O2c$, $\O2d$ and $\O3$ Reidemeister moves.}
\end{figure}\\
It follows directly from the presentation \eqref{eq:braid-presentation} of $\B B_m$ that the equivalence class of such diagrams depends on $\a$ and does not depend on the representative of $\a$, see for example \cite{KT}. It is called the \emph{closed braid} and is denoted by $\widehat{\a}$. It is straightforward to show that there is a one-to-one correspondence between the conjugacy classes in the braid groups $\B B_1,\B B_2,\B B_3,\ldots$ and closed braids, see for example \cite{KT}.

Now we are ready to describe our construction of the polynomial $I(a,z,t)$. We fix a natural number $k\geq 2$. Let $D$ be a diagram of an oriented link $L$. We remove from $D$ a small neighborhood of each crossing, see Figure \ref{fig:removement}.
%%%%%%%%%%%%%%%%%%%%%%%%%%%%%%%%%%%%%%%%%%
\begin{figure}[htb]
\centerline{\includegraphics[height=0.8in]{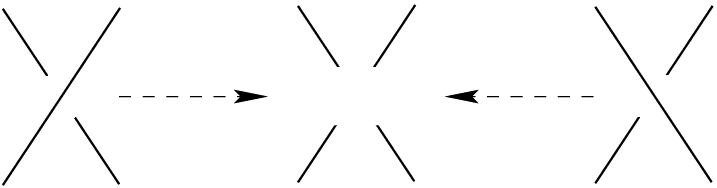}}
\caption{\label{fig:removement} Elimination of a neighborhood.}
\end{figure}
%%%%%%%%%%%%%%%%%%%%%%%%%%%%%%%%%%%%%%%%%%
The remaining arcs we will color by numbers from $\{1,\ldots,k\}$ according to the following rule: the adjacent arcs of each crossing we color as shown in Figure \ref{fig:arcs-coloring}a or in Figure \ref{fig:arcs-coloring}b. Note that in Figure \ref{fig:arcs-coloring}a we require that $p<q$ and in Figure \ref{fig:arcs-coloring}b we do not have this requirement, that is, $p$ can possibly be more than or equal to $q$. We also require that for every number in the set $\{1,\ldots,k\}$ there exists at least one arc which is colored by this number. We call a diagram $D$ with colored arcs a \emph{coloring} of $D$. Denote by $\C(D)$ the set of all colorings of $D$. We say that a crossing in a coloring of $D$ is \emph{special} if its adjacent arcs are colored as in Figure \ref{fig:arcs-coloring}a.
%%%%%%%%%%%%%%%%%%%%%%%%%%%%%%%%%%%%%%%%%%
\begin{figure}[htb]
\centerline{\includegraphics[height=2.2in]{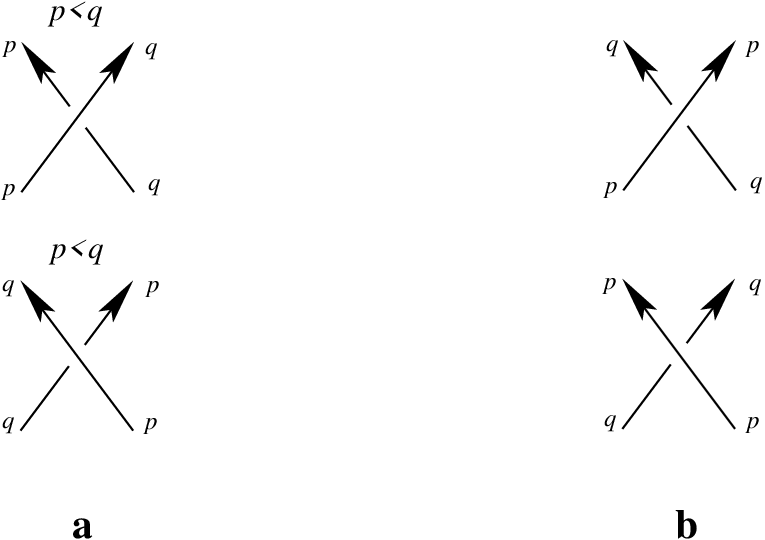}}
\caption{\label{fig:arcs-coloring} Coloring of arcs.}
\end{figure}
%%%%%%%%%%%%%%%%%%%%%%%%%%%%%%%%%%%%%%%%%%

Let $j\geq 0$. We denote by $\C(D)_j$ the set of all colorings of $D$ such that each coloring contains \emph{exactly} $j$ special crossings. Note that $\C(D)=\bigcup\limits_j\C(D)_j$. Let $C\in\C(D)_j$, then the \emph{sign} of $C$ is the product of the usual signs of the $j$ special crossings if $j>0$ and $+$ otherwise. We denote it by $\s(C)$. The coloring $C$ defines $k$ oriented links $L_1,\ldots,L_k$ as follows:
We smooth all $j$ special crossings as shown below, and denote by $L_i$ the oriented link whose diagram $D_i$ consists only of components colored by $i$. It is straightforward to see that the diagrams $D_i$ are subdiagrams of $D$ after the smoothing of $j$ special crossings, and so the links $L_i$ are well-defined.

\begin{equation*}
\ris{-4}{-3}{95}{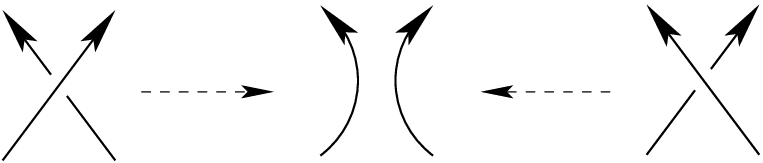}{-1.1}
\end{equation*}

\begin{ex}\rm
Let $k=2$. Figure \ref{fig:T-coloring} shows a coloring $C\in\C(D)_1$, where $D$ is a diagram of the trefoil, together with diagrams $D_1$ and $D_2$.
\end{ex}
%%%%%%%%%%%%%%%%%%%%%%%%%%%%%%%%%%%%%%%%%%
\begin{figure}[htb]
\centerline{\includegraphics[height=1.3in]{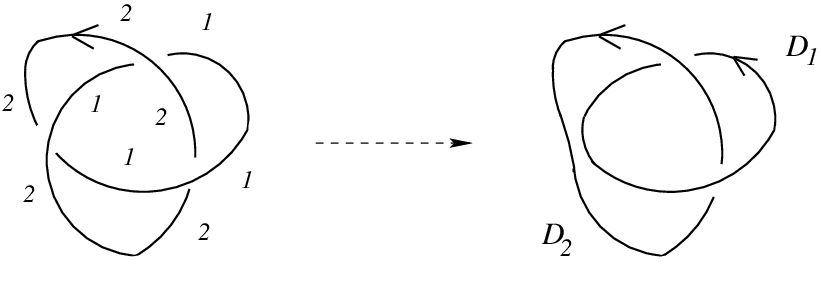}}
\caption{\label{fig:T-coloring} Coloring of the trefoil.}
\end{figure}
%%%%%%%%%%%%%%%%%%%%%%%%%%%%%%%%%%%%%%%%%%

\begin{defn}\rm
Denote by $\o(D)$ the \emph{writhe} of $D$, that is, the sum of signs of all crossings in $D$. We define
\begin{equation*}
I_k(D):=\sum_{j=0}^\infty z^j\sum_{C\in\C(D)_j}\s(C)a^{\o(D_1)}P(L_1)\cdot\ldots\cdot a^{\o(D_k)}P(L_k),
\end{equation*}
where $P(L_i)$ is the HOMFLY-PT polynomial of the link $L_i$, whose diagram $D_i$ is induced by a coloring $C$.
\end{defn}

We denote by $\c(D)$ and $\comp(D)$ the number of crossings and connected components in $D$ respectively. Note that if $j>\c(D)$ then by definition $\C(D)_j=\emptyset$. Hence $I_k$ is a well-defined Laurent polynomial.
Set
\begin{equation*}
I(D):=\sum_{k=2}^\infty I_k(D)t^k.
\end{equation*}
Let $k\geq 2$ and $C\in\C(D)$. Recall that by definition each arc of $C$ must be colored by some number from the set $\{1,\ldots,k\}$. Note that if $k$ is big enough, for example if $k> 4\c(D)+\comp(D)$, then $\C(D)=\emptyset$ and so $I_k(D)=0$. It follows that $I(D)$ is a well-defined 3-variable Laurent polynomial in variables $a,z,t$.
%%%%%%%%%%%%%%%%%%%%%%%%%%%%%%%%%%%%%%%%%%
\begin{figure}[htb]
\centerline{\includegraphics[width=3.5in]{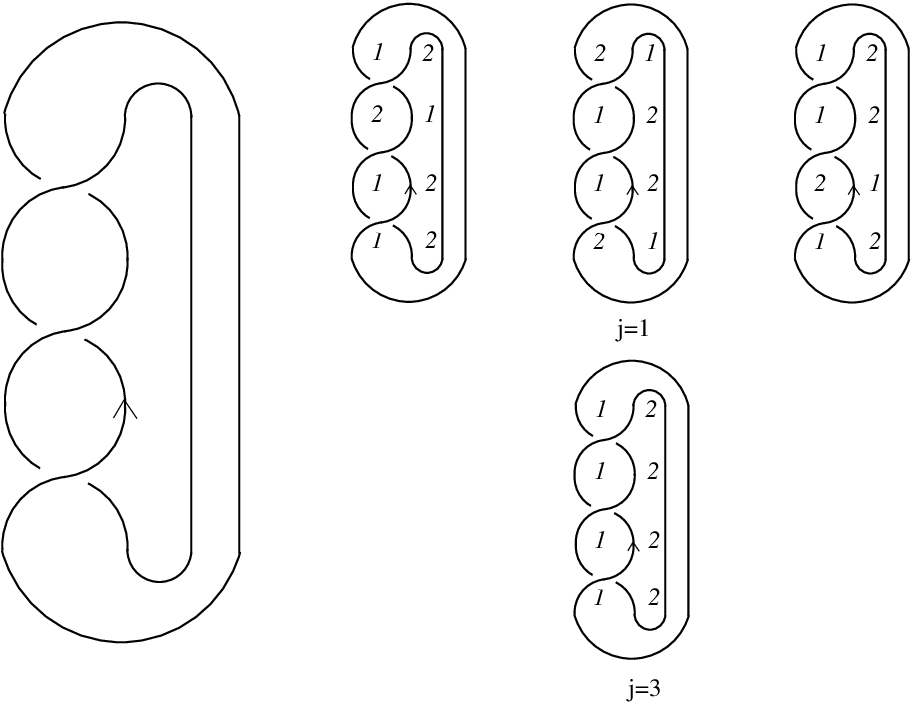}}
\caption{\label{fig:tref-diag-colorings} Diagram of the trefoil together with all possible colorings.}
\end{figure}
%%%%%%%%%%%%%%%%%%%%%%%%%%%%%%%%%%%%%%%%%%
\begin{ex}\label{ex:tref-computation}\rm
Let $D$ be an oriented diagram of the trefoil shown on the left of Figure \ref{fig:tref-diag-colorings}. Let us compute $I(D)$. If $k\geq 3$, then $I_k(D)=0$ because in this case the set of colorings $\C(D)$ is empty. Let $k=2$. If $j=0$, then $\C(D)_0=\emptyset$, because the rule of using all colors is violated. Let $j=1$. Then there are 3 possible colorings $C_1$, $C_2$ and $C_3$ of $D$ with $j=1$ special crossings. These colorings are shown on the top-right of Figure \ref{fig:tref-diag-colorings}. Each $C_i$ has a positive sign and it induces, as for example shown in Figure \ref{fig:T-coloring}, two standard diagrams of two unknots, colored by 1 and 2 respectively. Hence the contribution of each $C_i$ is $z$. Let $j=2$. Then $\C(D)_2=\emptyset$. Let $j=3$. Then there is only one possible coloring $C$ of $D$ with $j=3$ special crossings. This coloring has a positive sign and it is shown on the bottom-right of Figure \ref{fig:tref-diag-colorings}. In this case $C$ induces two diagrams of the unknot. Hence the contribution of $C$ is $z^3$. It follows that
$$I(D)=3zt^2+z^3t^2.$$
\end{ex}

\subsection{Main results}
\label{subsec-invariance}
Let us state our main results.

\begin{thm*} Let $\a\in \B B_m$ and let $D$ be any diagram of a closed braid $\widehat{\a}$. Then the polynomial
$I(D)$ is an invariant of the conjugacy class represented by $\a$.
\end{thm*}

Recall that the \emph{Conway
polynomial} $\n(L)$ is an invariant of an oriented link $L$. It is a polynomial in the
variable $z$, which satisfies the Conway skein relation:
\begin{equation}\label{eq:Conway-pol-skein}
\n\left(\ris{-4}{-1}{10}{L+}{-1.1}\right)-
\n\left(\ris{-4}{-1}{10}{L-}{-1.1}\right)=
z\n\left(\ris{-4}{-1.1}{10}{L0}{-1.1}\right).
\end{equation}
Let us state our second theorem.

\begin{thm*} Let $D$ be a diagram of an oriented link $L$. Then the polynomial
$$t^{-2}I(D)|_{a=1,t=0}-\o(D)z\n(L)$$
is independent of $D$ and hence is an invariant of the link $L$.
\end{thm*}

Now we recall the notion of a Conway triple of link diagrams.
Let $D_+$, $D_-$ and $D_0$ be a triple of link diagrams
which are identical except for a small fragment, where $D_+$ and
$D_-$ have  a positive and a negative  crossing respectively, and
$D_0$ has a smoothed crossing, see Figures \ref{fig:triple}a and \ref{fig:triple}b.
%%%%%%%%%%%%%%%%%%%%%%%%%%%%%%%%%%%%%%%%%
\begin{figure}[htb]
\centerline{\includegraphics[height=0.9in]{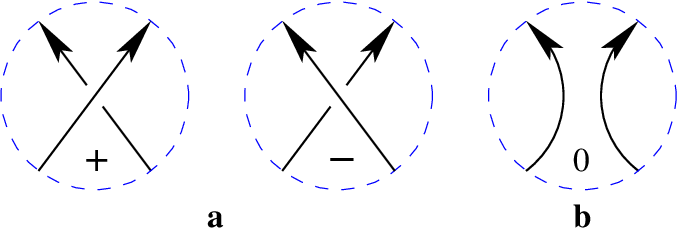}}
\caption{\label{fig:triple} Conway triple.}
\end{figure}
%%%%%%%%%%%%%%%%%%%%%%%%%%%%%%%%%%%%%%%%%
Such a triple of link diagrams is called a \emph{Conway triple}.

\begin{thm*}\label{thm:Conway-skein}
The polynomial $I$ satisfies the Conway skein relation, that is, for each Conway triple of link diagrams $D_+$, $D_-$ and $D_0$ we have
$$I(D_+)-I(D_-)=zI(D_0).$$
\end{thm*}

Here we recall the notion of real-valued Vassiliev or finite type invariants of braids. This is a straightforward modification of real-valued finite type link invariants, see \cite{BN1,V2,V1}. Let $v\colon\B B_m\to \B R$ be a real-valued invariant of braids. In the same way as knots are extended to singular knots one can extend the braid group $\B B_m$ to the singular braid monoid $\B{SB}_m$ of singular braids on $m$ stings \cite{Bir}. We extend $v$ to singular braids by using the recursive rule
\begin{equation}\label{eq:FTI-property}
v\left(\ris{-4}{-1}{10}{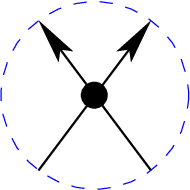}{-1.1}\right)=
v\left(\ris{-4}{-1}{10}{L+}{-1.1}\right)-v\left(\ris{-4}{-1}{10}{L-}{-1.1}\right).
\end{equation}
The picture on the left hand side represents a small neighborhood of a singular point in a
singular braid. Those on the right hand side represent the braids which are obtained from the
previous one by a positive and negative resolution of that singular point. The singular braids on the right hand side
have one singular point less then the braid on the left hand side.

An invariant $v$ is said to be of \emph{finite type}
or \emph{Vassiliev}, if for some $n\in\B N$ it vanishes for any singular braid $\a\in \B{SB}_m$ with more than $n$ singular points. A minimal such $n$ is called the degree of $v$. As a corollary of Theorem \ref{thm:Conway-skein} we have

\begin{cor}
The $n$-th coefficient of the polynomial $(t^{-k}I)|_{a=1,t=0}$, where $k\geq 2$, is a finite type invariant of braids of degree $n$.
\end{cor}
\begin{proof}
Let $k\geq 2$. Note that $P|_{a=1}=\n$, hence by definition we have
$$(t^{-k}I)|_{a=1,t=0}=\sum_{j=0}^\infty z^j\sum_{C\in\C(D)_j}\s(C)\n(L_1)\cdot\ldots\cdot \n(L_k),$$
which is a polynomial in the variable $z$. It follows from Theorem 3 that $(t^{-k}I)|_{a=1,t=0}$ satisfies the Conway skein relation, and hence by the same argument as in the proof for the coefficients of the Conway polynomial $\n$, see for example \cite[Page 10]{BN1}, it $n$-th coefficient is a finite type invariant of degree $n$.
\end{proof}

Let $P(L)$ be the HOMFLY-PT polynomial of a link $L$.
We denote by $P'_a(L)$ the first partial derivative of $P(L)$ w.r.t. $a$.
Then $zP'_a(L)|_{a=1}$ is a polynomial in the variable $z$.

\begin{thm*}\label{thm:identification}
Let $D$ be any diagram of a link $L$. Then
\begin{equation*}
t^{-2}I(D)|_{a=1,t=0}-w(D)z\n(L)=zP'_a(L)|_{a=1}.
\end{equation*}
\end{thm*}

As a corollary we have
\begin{cor}
The $n$-th coefficient of the polynomial $(t^{-2}I)|_{a=1,t=0}-wz\n$ is a finite type link invariant of degree $n$.
\end{cor}
\begin{proof}
The polynomial $zP'_a(L)|_{a=1}$ satisfies the following skein relation
$$zP'_a(L_+)|_{a=1}-zP'_a(L_-)|_{a=1}=z(zP'_a(L_0)|_{a=1}-\n(L_+)-\n(L_-)).$$
It follows that it $n$-th coefficient is a finite type invariant of degree $n$.
\end{proof}

\begin{ex}\rm
Let $D$ be a diagram of the trefoil $T$ shown in Figure \ref{fig:tref-diag-colorings}. Note that $w(D)=3$ and $\n(T)=1+z^2$. In Example \ref{ex:tref-computation} we showed that $t^{-2}I(D)|_{a=1,t=0}=3z+z^3$. Hence
$$t^{-2}I(D)|_{a=1,t=0}-w(D)z\n(T)=-2z^3,$$
and this coincides with the fact that $zP'_a(T)|_{a=1}=-2z^3$.
\end{ex}

\begin{rem}\rm
Let $G$ be a Gauss diagram of $L$ (for a precise definition see for example \cite{GPV,PV}). Another interpretation of the polynomial $zP'_a(L)|_{a=1}$ in terms of counting surfaces with two boundary components in $G$ was given recently by the author in \cite{B1,B2}.
\end{rem}

\section{Proofs}

\subsection{Invariance under certain Reidemeister moves.}

Let $D$ be a link diagram and $k\geq 2$. For each coloring $C\in\C(D)_j$ set
\begin{equation*}
I_{k,j}(D)_C:=\s(C)z^ja^{\o(D_1)}P(L_1)\cdot\ldots\cdot a^{\o(D_k)}P(L_k),
\end{equation*}
where $L_i$ is the link induced by $C$, as explained in Subsection \ref{subsec-construction}.
At the beginning we will prove the following useful

\begin{prop}\label{prop:O2}
The polynomial $I$ is invariant under $\O2a$ and $\O2b$ moves shown in Figure \ref{fig:Reid-moves}.
\end{prop}

\begin{proof}
Recall, that by definition the polynomial $I:=\sum\limits_{k=2}^\infty I_k t^k$. Hence it is enough to prove the statement for the polynomials $I_k$.

Let $k=2$, and $D$ and $\wD$ be two diagrams which differ by an application of one $\O2a$ move such that $\c(\wD)=\c(D)+2$.
For each $j\geq 0$ and $0\leq d\leq 2$ denote by $\C(\wD)_{j,d}$ the subset of $\C(\wD)_j$ which contains all colorings with $d$ \emph{special crossings} in the distinguished fragment of $\wD$.
%%%%%%%%%%%%%%%%%%%%%%%%%%%%%%%%%%%%%%%%%%
\begin{figure}[htb]
\centerline{\includegraphics[height=2.2in]{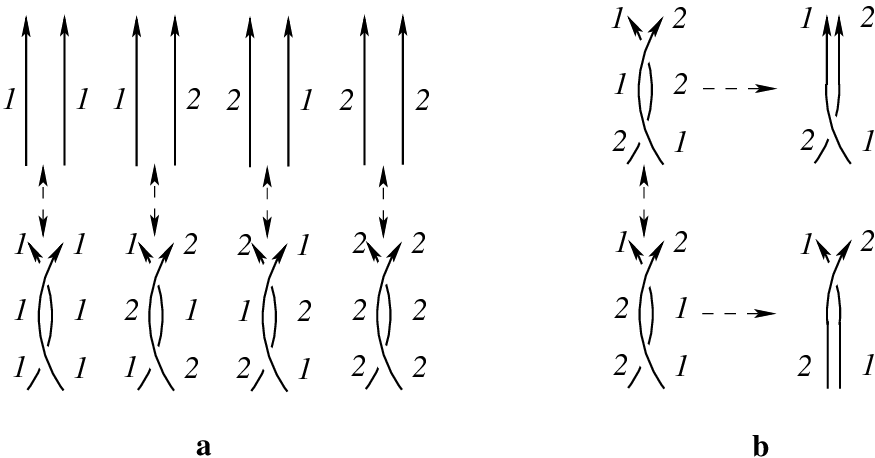}}
\caption{\label{fig:Rad2a-inv} Invariance under $\O2a$ move.}
\end{figure}
%%%%%%%%%%%%%%%%%%%%%%%%%%%%%%%%%%%%%%%%%%
\begin{figure}[htb]
\centerline{\includegraphics[height=2.2in]{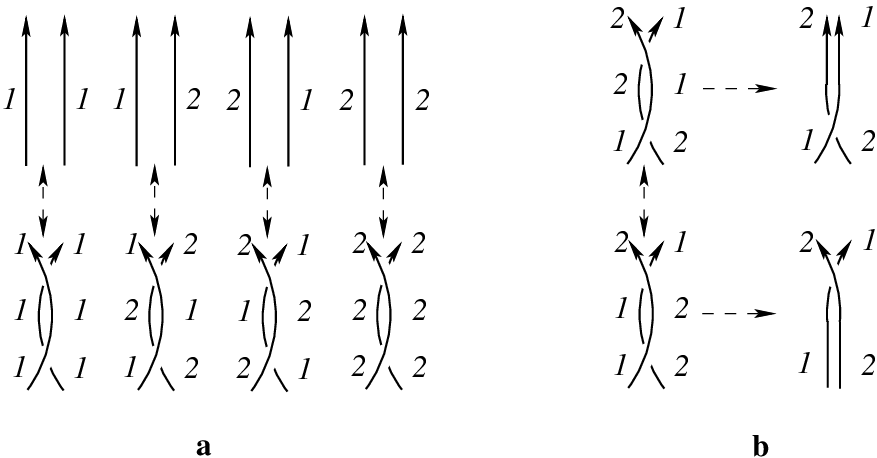}}
\caption{\label{fig:Rad2b-inv} The invariance under $\O2b$ move.}
\end{figure}
%%%%%%%%%%%%%%%%%%%%%%%%%%%%%%%%%%%%%%%%%%
Let $j\geq 0$. We have a bijection between the sets $\C(D)_j$ and  $\C(\wD)_{j,0}$ as shown in Figure \ref{fig:Rad2a-inv}a. We assume here and further on that the coloring of arcs away from the distinguished fragments is identical for two corresponding colorings $C$ and $\wC$. For each two corresponding colorings $C\in\C(D)_j$ and $\wC\in \C(\wD)_{j,0}$ the diagrams $D_1$ and $\wD_1$ as well as the diagrams $D_2$ and $\wD_2$ are isotopic. We also have $\s(C)=\s(\wC)$, $\o(D_1)=\o(\wD_1)$ and $\o(D_2)=\o(\wD_2)$. Hence
\begin{equation}\label{eq:Rad2-case1}
I_{2,j}(D)_C=I_{2,j}(\wD)_{\wC}.
\end{equation}
Note that by definition the set $\C(\wD)_{j,1}$ contains all colorings of a diagram $\wD$ which have exactly $j$ special crossings, so that the distinguished fragment contains exactly one special crossing. For each $\widetilde{C'}\in \C(\wD)_{j,1}$ there exists a corresponding coloring $\widetilde{C''}\in \C(\wD)_{j,1}$ and vice-versa. This correspondence is shown in Figure \ref{fig:Rad2a-inv}b. In this case $\s(\widetilde{C'})=-\s(\widetilde{C''})$ and hence
\begin{equation}\label{eq:Rad2-case2}
\sum_{\wC\in\C(\wD)_{j,1}}I_{2,j}(\wD)_{\wC}=0.
\end{equation}
Since there could be no coloring with exactly two special crossings in the distinguished fragment of $\wD$, because it will violate the rule of Figure \ref{fig:arcs-coloring}, we have $\C(\wD)_{j,2}=\emptyset$ for each $j$. Combining this statement with equations \eqref{eq:Rad2-case1} and \eqref{eq:Rad2-case2} we obtain
\begin{equation*}
I_2(D)=\sum_{j=0}^\infty\sum_{C\in\C(D)_j}I_{2,j}(D)_C=
\sum_{j=0}^\infty\sum_{d=0}^2\sum_{\wC\in\C(\wD)_{j,d}}I_{2,j}(\wD)_{\wC}=I_2(\wD).
\end{equation*}

The invariance of $I_2$ under $\O2b$ move is proved similarly. The correspondence of colorings is summarized in Figure \ref{fig:Rad2b-inv}.

The proof of the invariance of the polynomials $I_k$ ($k\geq 3$) under $\O2a$ and $\O2b$ moves is very similar and is left to the reader
\end{proof}

Now we try to understand how $\O2c$ and $\O2d$ moves affect the polynomial $I_2$. Let $D$ and $\wD$ be two diagrams which differ by an application of one $\O2c$ move, such that $\c(\wD)=\c(D)+2$. In this case, there is a bijection between the sets $\C(D)_j$ and  $\C(\wD)_{j,0}$ as shown in Figure \ref{fig:Rad2c-inv}a.
%%%%%%%%%%%%%%%%%%%%%%%%%%%%%%%%%%%%%%%%%%
\begin{figure}[htb]
\centerline{\includegraphics[height=2.2in]{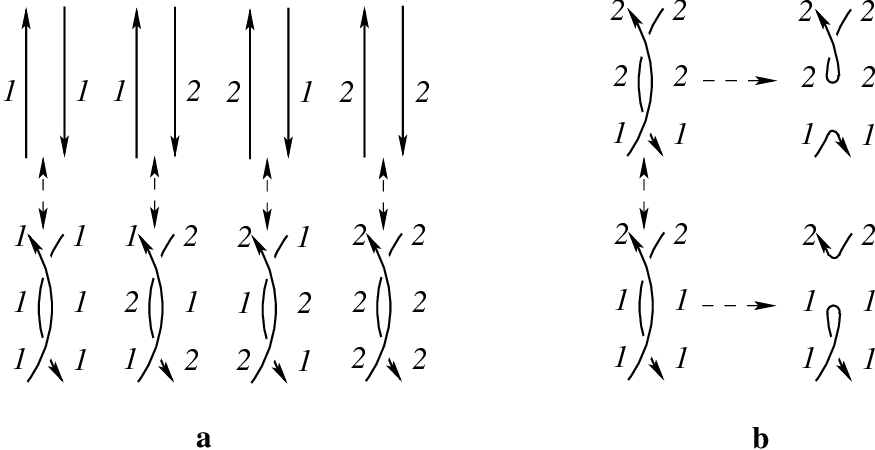}}
\caption{\label{fig:Rad2c-inv} Invariance under $\O2c$ move.}
\end{figure}
%%%%%%%%%%%%%%%%%%%%%%%%%%%%%%%%%%%%%%%%%%
For each two corresponding colorings $C\in\C(D)_j$ and $\wC\in \C(\wD)_{j,0}$ the diagrams $D_1$ and $\wD_1$ as well as diagrams $D_2$ and $\wD_2$ are isotopic. We also have $\s(C)=\s(\wC)$, $\o(D_1)=\o(\wD_1)$ and $\o(D_2)=\o(\wD_2)$. Hence
\begin{equation}\label{eq:Rad2c-case1}
I_{2,j}(D)_C=I_{2,j}(\wD)_{\wC}.
\end{equation}
For each $\widetilde{C'}\in \C(\wD)_{j,1}$ there exists a corresponding coloring $\widetilde{C''}\in \C(\wD)_{j,1}$ and vice-versa. This correspondence is shown in Figure \ref{fig:Rad2c-inv}b ($\widetilde{C'}$ is on the top and $\widetilde{C''}$ is on the bottom).  Let $\wD_1'$, $\wD_2'$, $\wD_1''$ and $\wD''_2$ be the diagrams induced by $\widetilde{C'}$ and $\widetilde{C''}$ respectively. Let $L_1$ and $L_2$ be the links whose diagrams are  $\wD_1'$, $\wD''_1$ and $\wD'_2$, $\wD''_2$ respectively.  It follows that
\begin{align*}
&I_{2,j}(\wD)_{\widetilde{C'}}=\s(\widetilde{C'})z^ja^{\o(\wD'_1)}P(L_1)\cdot a^{\o(\wD'_2)}P(L_2),\\
&I_{2,j}(\wD)_{\widetilde{C''}}=\s(\widetilde{C''})z^ja^{\o(\wD''_1)}P(L_1)\cdot a^{\o(\wD''_2)}P(L_2).
\end{align*}
Note that in this case $\s(\widetilde{C'})=-\s(\widetilde{C''})$, $\o(\wD'_1)=\o(\wD''_1)+1$ and $\o(\wD'_2)=\o(\wD''_2)+1$. It follows that
\begin{equation}\label{eq:Rad2c-case2}
I_{2,j}(\wD)_{\widetilde{C'}}+I_{2,j}(\wD)_{\widetilde{C''}}=
\s(\widetilde{C'})z^j a^{\o(\wD'_1)+\o(\wD'_2)}P(L_1)\cdot P(L_2)(1-a^{-2}).
\end{equation}
Combining equations \eqref{eq:Rad2c-case1} and \eqref{eq:Rad2c-case2} we obtain
\begin{equation*}
I_2(\wD)-I_2(D)=0\quad\textrm{only if}\quad a=\pm1.
\end{equation*}

In the case of $\O2d$ move a similar analysis shows that
\begin{equation*}
I_2(\wD)-I_2(D)=0\quad\textrm{only if}\quad a=\pm1.
\end{equation*}
The correspondence of colorings is shown in Figure \ref{fig:Rad2d-inv}. Hence we proved the following corollary
\begin{cor}\label{cor:O2cd}
The polynomials $I_2(\pm1,z)$ are invariant under $\O2a$, $\O2b$, $\O2c$ and $\O2d$ moves.
\end{cor}

%%%%%%%%%%%%%%%%%%%%%%%%%%%%%%%%%%%%%%%%%%
\begin{figure}[htb]
\centerline{\includegraphics[height=2.2in]{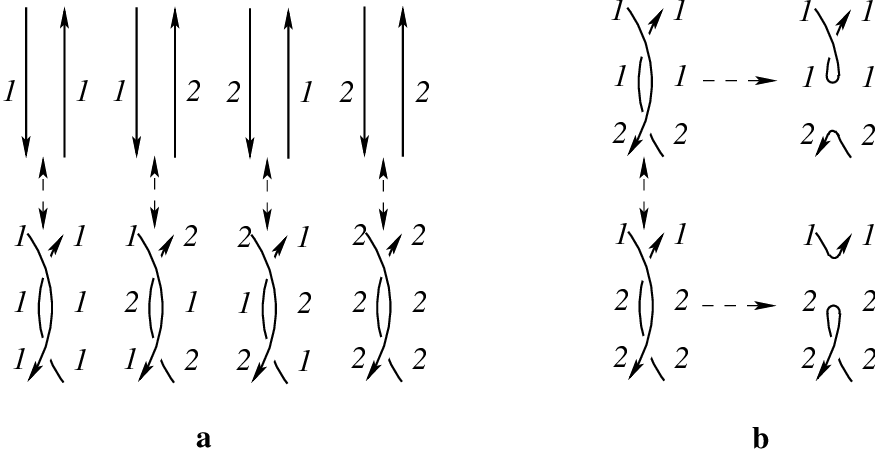}}
\caption{\label{fig:Rad2d-inv} Invariance under $\O2d$ move.}
\end{figure}
%%%%%%%%%%%%%%%%%%%%%%%%%%%%%%%%%%%%%%%%%%

\begin{rem}\rm
Polynomial $I_2(1,z)$ determines the polynomial $I_2(-1,z)$ and vice-versa, hence further we will not discuss the polynomial $I_2(-1,z)$.
By inspecting the definition of $I_k(D)$ we see that
$$t^{-k}I(a,z,t)(D)|_{a=1,t=1}=I_k(1,z)(D)=\sum_{j=0}^\infty z^j\sum\limits_{C\in\C(D)_j}\s(C)\n(L_1)\cdot\ldots\cdot\n(L_k).$$
\end{rem}

\begin{prop}\label{prop:O3}
The polynomial $I_2$ is invariant under $\O3$ Reidemeister move shown in Figure \ref{fig:Reid-moves}.
\end{prop}

\begin{proof}

Let $D$ and $\wD$ be two diagrams which differ by an application of one $\O3$ move.
For each $j\geq 0$ and $d\geq 0$ denote by $\C(D)_{j,d}$ and $\C(\wD)_{j,d}$ the subsets of $\C(D)_j$ and $\C(\wD)_j$ respectively which contain all colorings with $d$ \emph{special crossings} in the distinguished fragments of $D$ and $\wD$. Note that the number of colors is $k=2$, hence for $d>2$ the sets $\C(D)_{j,d}$ and $\C(\wD)_{j,d}$ are empty.\\
\textbf{Case 1.} Let $d=0$. Note that if $j=0$, then $d=0$ and in this case we have $\C(D)_0=\C(D)_{0,0}$ and $\C(\wD)_0=\C(\wD)_{0,0}$. There is a bijection between the sets $\C(D)_{j,0}$ and $\C(\wD)_{j,0}$. This bijection is shown in Figure \ref{fig:Rad3-0-inv}. HOMFLY-PT polynomial is a link invariant and hence for each pair of corresponding colorings $C$ and $\wC$ we have
$$I_{2,j}(D)_C=I_{2,j}(\wD)_{\wC},$$
and hence
\begin{equation}\label{eq:d=0}
\sum_{C\in\C(D)_{j,0}}I_{2,j}(D)_C=\sum_{\wC\in\C(\wD)_{j,0}}I_{2,j}(\wD)_{\wC}.
\end{equation}
%%%%%%%%%%%%%%%%%%%%%%%%%%%%%%%%%%%%%%%%%%
\begin{figure}[htb]
\centerline{\includegraphics[height=1.6in]{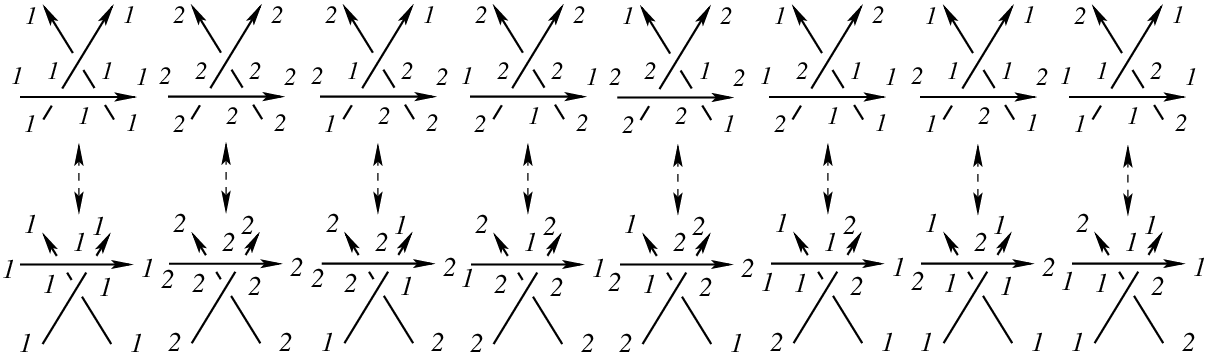}}
\caption{\label{fig:Rad3-0-inv} Correspondence of colorings with zero special crossings.}
\end{figure}
%%%%%%%%%%%%%%%%%%%%%%%%%%%%%%%%%%%%%%%%%
\begin{figure}[htb]
\centerline{\includegraphics[height=3.55in]{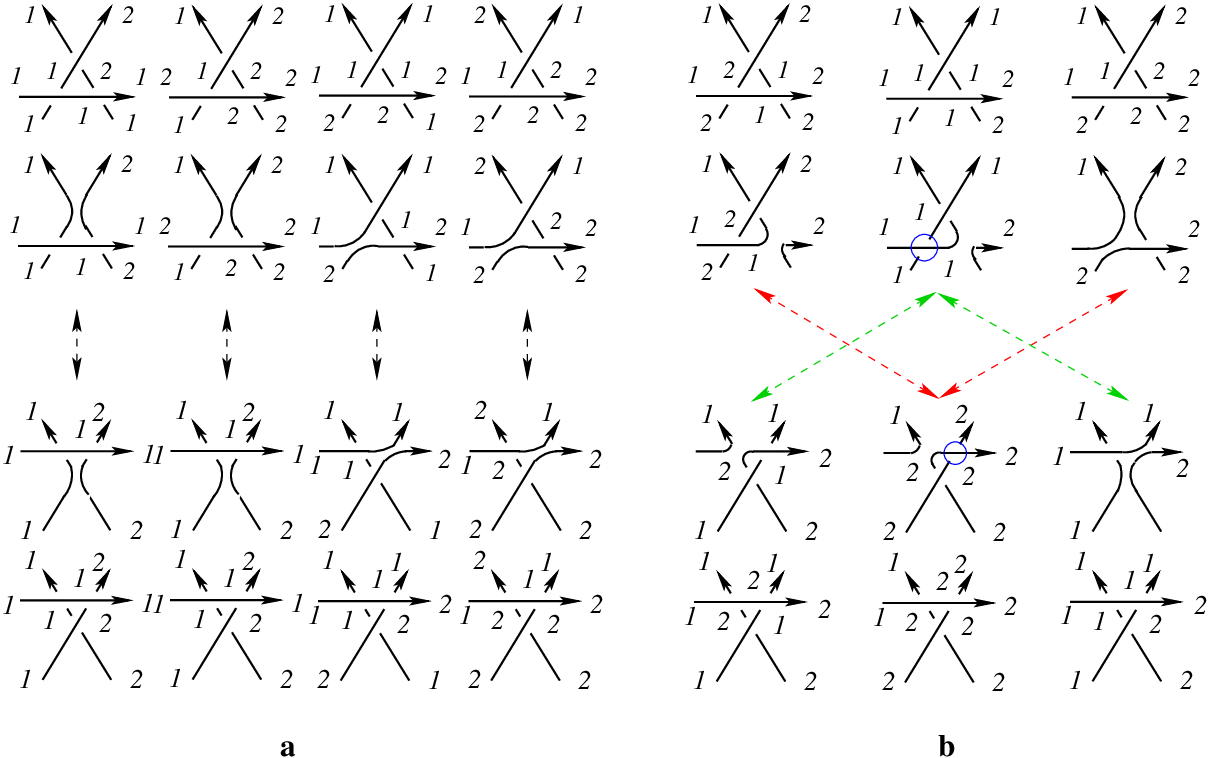}}
\caption{\label{fig:Rad3-12-inv} Correspondence between $\C(D)_{j,1}\cup \C(D)_{j+1,2}$ and $\C(\wD)_{j,1}\cup \C(\wD)_{j+1,2}$.}
\end{figure}

\textbf{Case 2.} Let $j\geq 1$ and $d=1,2$. There is a correspondence between the sets $\C(D)_{j,1}\cup \C(D)_{j+1,2}$ and $\C(\wD)_{j,1}\cup \C(\wD)_{j+1,2}$. It is shown in Figure \ref{fig:Rad3-12-inv}. In the top row we present colorings of distinguished fragment of diagrams in the set
$\C(D)_{j,1}\cup \C(D)_{j+1,2}$ and in the bottom row we present colorings of distinguished fragment of diagrams in the set $\C(\wD)_{j,1}\cup \C(\wD)_{j+1,2}$. Outside of these fragments the colorings of corresponding diagrams are the same. The second and third row from the top represent distinguished fragments of $D$ and $\wD$ respectively, after smoothing of all special crossings in these fragments. Now we are going to discuss this correspondence in greater detail.

Each pair of corresponding colorings $C$ and $\wC$ in Figure \ref{fig:Rad3-12-inv}a belongs to the sets $\C(D)_{j,1}$ and $\C(\wD)_{j,1}$ respectively. The induced diagrams $D_i$ and $\wD_i$ are isotopic for $i=1,2$. Hence
$$I_{2,j}(D)_{C}=I_{2,j}(\wD)_{\wC}.$$

The correspondence of the remaining colorings is much more complicated and it is shown in Figure \ref{fig:Rad3-12-inv}b. Let us denote by $C_l$ (respectively by $\wC_l$), $C_c$ (respectively by $\wC_c$) and $C_r$ (respectively by $\wC_r$) the colorings in the set $\C(D)_{j,1}\cup \C(D)_{j+1,2}$ (respectively in the set $\C(\wD)_{j,1}\cup \C(\wD)_{j+1,2}$) whose fragment is shown on the top-left (respectively on the bottom-left), top-center (respectively on the bottom-center) and top-right (respectively on the bottom-right) of Figure \ref{fig:Rad3-12-inv}b. Note that $C_l,C_c\in \C(D)_{j,1}$, $\wC_l,\wC_c\in \C(\wD)_{j,1}$ and $C_r\in \C(D)_{j+1,2}$, $\wC_r\in \C(\wD)_{j+1,2}$. We will show that
\begin{equation}\label{eq:compl-cor1}
\sum_{C_c\in\C(D)_{j,1}}I_{2,j}(D)_{C_c}=\sum_{\wC_l\in\C(\wD)_{j,1}}I_{2,j}(\wD)_{\wC_l}
+\sum_{\wC_r\in\C(\wD)_{j+1,2}}I_{2,j+1}(\wD)_{\wC_r}
\end{equation}

\begin{equation}\label{eq:compl-cor2}
\sum_{\wC_c\in\C(\wD)_{j,1}}I_{2,j}(\wD)_{\wC_c}=\sum_{C_l\in\C(D)_{j,1}}I_{2,j}(D)_{C_l}
+\sum_{C_r\in\C(D)_{j+1,2}}I_{2,j+1}(D)_{C_r}.
\end{equation}

We start with the proof of \eqref{eq:compl-cor1}. The corresponding colorings $C_c$, $\wC_l$ and $\wC_r$  (the correspondence is shown by green arrows) induce link diagrams $D_{1c}$, $D_{2c}$, $\wD_{1l}$, $\wD_{2l}$ and $\wD_{1r}$, $\wD_{2r}$. It is shown in Figure \ref{fig:Rad3-12-inv}b that the diagrams $D_{2c}$, $\wD_{2l}$ and $\wD_{2r}$ are isotopic. We have
\begin{align*}
&a^{\o(D_{1c})}P(D_{1c})-a^{\o(\wD_{1l})}P(\wD_{1l})-za^{\o(\wD_{1r})}P(\wD_{1r}):=\\
&a^{\o(D_{1c})}P\left(\ris{-4}{-1}{10}{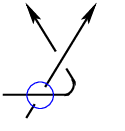}{-1.1}\right)-a^{\o(\wD_{1l})}P\left(\ris{-4}{-1}{8}{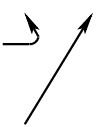}{-1.1}\right)-
za^{\o(\wD_{1r})}P\left(\ris{-5}{-1}{9}{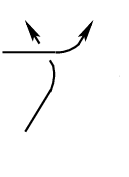}{-1.1}\right)=\\
&a^{\o(D_{1c})-1}\left(aP\left(\ris{-4}{-1}{10}{conway1}{-1.1}\right)
-a^{-1}P\left(\ris{-4}{-1}{10}{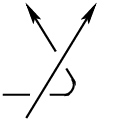}{-1.1}\right)-zP\left(\ris{-5}{-1}{9}{conway3}{-1.1}\right)\right)=0
\end{align*}
The second equality follows from the fact that $P$ is invariant under the second Reidemeister move, and the third equality is the HOMFLY-PT skein relation \eqref{eq:Homfly-skein} applied to the blue crossing in $D_{1c}$. All crossings in the distinguished fragments of $D$ and $\wD$ are positive. It follows that
$$\s(C_c)=\s(\wC_l)=\s(\wC_r).$$
This yields
\begin{align*}
&I_{2,j}(D)_{C_c}:=\s(C_c)z^ja^{\o(D_{1c})}P(D_{1c})\cdot a^{\o(D_{2c})}P(D_{2c})=\\
&\s(\wC_l)z^ja^{\o(\wD_{1l})}P(\wD_{1l})\cdot a^{\o(\wD_{2l})}P(\wD_{2l})+\\
&\s(\wC_r)z^{j+1}a^{\o(\wD_{1r})}P(\wD_{1r})\cdot a^{\o(\wD_{2r})}P(\wD_{2r}):=
I_{2,j}(D)_{\wC_l}+I_{2,j+1}(D)_{\wC_r},
\end{align*}
and equation \eqref{eq:compl-cor1} follows.

The proof of \eqref{eq:compl-cor2} is very similar to the proof of \eqref{eq:compl-cor1}. In this case the corresponding colorings $\wC_c$, $C_l$ and $C_r$  (the correspondence is shown by red arrows) induce link diagrams $\wD_{1c}$, $\wD_{2c}$, $D_{1l}$, $D_{2l}$ and $D_{1r}$, $D_{2r}$. It is shown in Figure \ref{fig:Rad3-12-inv}b that the diagrams $\wD_{1c}$, $D_{1l}$ and $D_{1r}$ are isotopic. In this case
\begin{align*}
&a^{\o(\wD_{2c})}P(\wD_{2c})-a^{\o(D_{2l})}P(D_{2l})-za^{\o(D_{2r})}P(D_{2r}):=\\
&a^{\o(\wD_{2c})}P\left(\ris{-4}{-1}{10}{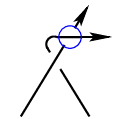}{-1.1}\right)-
a^{\o(D_{2l})}P\left(\ris{-4}{0}{9}{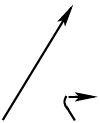}{-1.1}\right)-za^{\o(D_{2r})}P\left(\ris{-4}{-1}{9}{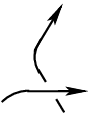}{-1.1}\right)=\\
&a^{\o(\wD_{2c})-1}\left(aP\left(\ris{-4}{-1}{10}{conway4}{-1.1}\right)-a^{-1}P\left(\ris{-4}{0}{9}{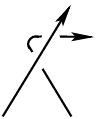}{-1.1}\right)
-zP\left(\ris{-4}{-1}{9}{conway5}{-1.1}\right)\right)=0.
\end{align*}
Note that $\s(\wC_c)=\s(C_l)=\s(C_r)$ and equation \eqref{eq:compl-cor2} follows.

We combine equations \eqref{eq:d=0}, \eqref{eq:compl-cor1} and \eqref{eq:compl-cor2} and obtain
\begin{align*}
&I_2(D):=\sum_{j=0}^\infty\sum_{C\in\C(D)_j}I_{2,j}(D)_C=\sum_{j=0}^\infty\sum_{C\in\C(D)_{j,0}}I_{2,j}(D)_C+
\sum_{j=1}^\infty\sum_{C\in\C(D)_{j,1}}I_{2,j}(D)_C+\\
&\sum_{j=2}^\infty\sum_{C\in\C(D)_{j,2}}I_{2,j}(D)_C=
\sum_{j=0}^\infty\sum_{\wC\in\C(\wD)_{j,0}}I_{2,j}(\wD)_{\wC}+
\sum_{j=1}^\infty\sum_{\wC\in\C(\wD)_{j,1}}I_{2,j}(\wD)_{\wC}+\\
&\sum_{j=2}^\infty\sum_{\wC\in\C(\wD)_{j,2}}I_{2,j}(\wD)_{\wC}=
\sum_{j=0}^\infty\sum_{\wC\in\C(\wD)_j}I_{2,j}(\wD)_{\wC}:=I_2(\wD).
\end{align*}
This concludes the proof of the proposition.
\end{proof}

\begin{rem}\rm
The polynomials $I_2(D)$ and $I_2(1,z)(D)$ are not invariant under $\O1a$ and $\O1b$ Reidemeister moves shown in Figure \ref{fig:Reid-moves}. Let $D$, $D'$ and $D''$ diagrams of the unknot shown in Figures \ref{fig:unknot-diagrams}a, \ref{fig:unknot-diagrams}b and \ref{fig:unknot-diagrams}c respectively. Then $I_2(1,z)(D)=I_2(D)=0$, but $I_2(1,z)(D')=I_2(D')=z$ and $I_2(1,z)(D'')=I_2(D'')=-z$.
\end{rem}
%%%%%%%%%%%%%%%%%%%%%%%%%%%%%%%%%%%%%%%%%%
\begin{figure}[htb]
\centerline{\includegraphics[height=0.95in]{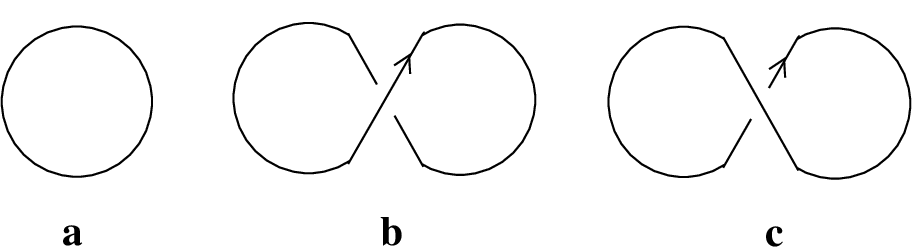}}
\caption{\label{fig:unknot-diagrams} Diagrams of the unknot.}
\end{figure}
%%%%%%%%%%%%%%%%%%%%%%%%%%%%%%%%%%%%%%%%%%

\subsection{Proof of Theorem 2}

It follows from the work of Polyak \cite{P} that in order to prove the invariance of $t^{-2}I(a,z,t)|_{a=1,t=0}-\o z\n$ it is enough to prove its invariance under $\O1a$, $\O1b$, $\O2c$, $\O2d$ and $\O3$ moves.

We know that the writhe $\o$ and the Conway polynomial $\n$ are invariant under $\O2c$, $\O2d$ and $\O3$ moves. Note that by definition
\begin{equation}\label{I-2}
(t^{-2}I(a,z,t))|_{a=1,t=0}=I_2(1,z).
\end{equation}

It follows from Corollary \ref{cor:O2cd} and Proposition \ref{prop:O3} that it is enough to prove the invariance
of $I_2(1,z)-\o z\n$ under $\O1a$ and $\O1b$ moves.

Let $D$ and $\wD$ be two diagrams which differ by an application of $\O1a$ move such that $\c(\wD)=\c(D)+1$. For each $j\geq 0$ and $d=0,1$ denote by $\C(\wD)_{j,d}$ the subset of $\C(\wD)_j$ which contains all colorings with $d$ \emph{special crossings} in the distinguished fragment of $\wD$. Recall that
$$I_2(1,z)(D)=\sum_{j=0}^\infty z^j\sum\limits_{C\in\C(D)_j}\s(C)\n(L_1)\cdot\n(L_2).$$
The Conway polynomial is invariant under $\O1a$ move and hence
$$I_2(1,z)(D)=\sum_{j=0}^\infty \sum_{\wC\in\C(\wD)_{j,0}}I_{2,j}(1,z)(\wD)_{\wC}.$$
For each $\wC\in\C(\wD)_{j,1}$ the coloring of arcs in the distinguished fragment is shown below.
$$\ris{0}{0}{15}{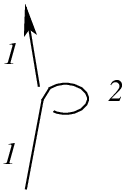}{-1.1}$$
We have
\begin{align*}
&I_2(1,z)(\wD)-I_2(1,z)(D)=\sum_{j=1}^\infty\sum_{\wC\in\C(\wD)_{j,1}}I_2(1,z)(\wD)_{\wC}=\\
&\sum_{\wC\in\C(\wD)_{1,1}}I_2(1,z)(\wD)_{\wC}=z\sum_{\wC\in\C(\wD)_{1,1}}\n(L_1)\cdot\n(L_2)=z\n(L).
\end{align*}
The second equality follows from the fact that Conway polynomial of split links equals zero, the third equality is a definition and the fourth equality follows from the fact that $L_2$ is the unknot and $L_1$ is a link $L$.
Note that
$$w(\wD)z\n(L)-w(D)z\n(L)=z\n(L),$$
and thus
$$I_2(1,z)(\wD)-w(\wD)z\n(L)-(I_2(1,z)(D)-w(D)z\n(L))=0.$$

The proof of the invariance of $I_2(1,z)-\o z\n$ under $\O1b$ move is very similar and is left to the reader.
\qed

\subsection{Proof of Theorem 1}
It follows from Propositions \ref{prop:O2} and \ref{prop:O3} that it is enough to prove the invariance of the polynomials $I_k$ ($k\geq 3$) under $\O3$ move.

Let $k\geq 3$, and $D$ and $\wD$ be two diagrams which differ by an application of one $\O3$ move.
For each $j\geq 0$ and $d\geq0$ denote by $\C(D)_{j,d,3}$ and $\C(\wD)_{j,d,3}$ the subsets of $\C(D)_{j,d}$ and $\C(\wD)_{j,d}$ respectively (these sets were defined in the proof of Proposition \ref{prop:O3}) which contain all colorings with $j$ \emph{special crossings}, and \emph{exactly} $d$ \emph{special crossings} and \emph{exactly} 3 \emph{different colors} in the distinguished fragments of $D$ and $\wD$. Note that when $d>3$ the sets $\C(D)_{j,d}$ and $\C(\wD)_{j,d}$ are empty. We start with the model case when $k=3$.
%%%%%%%%%%%%%%%%%%%%%%%%%%%%%%%%%%%%%%%%%%%
\begin{figure}[htb]
\centerline{\includegraphics[height=2in]{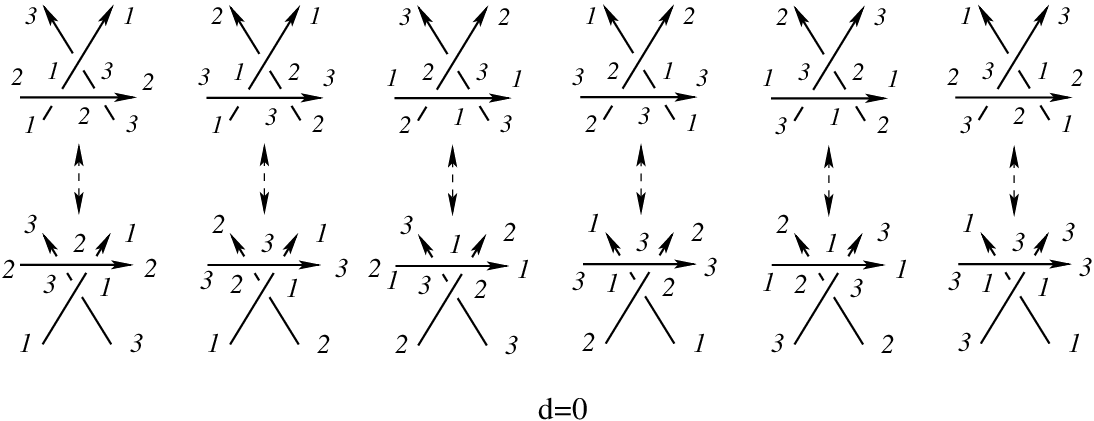}}
\caption{\label{fig:Rad3-3col0-inv} Correspondence of colorings with 0 special crossings and exactly 3 different colors in the fragments.}
\end{figure}
%%%%%%%%%%%%%%%%%%%%%%%%%%%%%%%%%%%%%%%%%%%
\begin{figure}[htb]
\centerline{\includegraphics[height=3.5in]{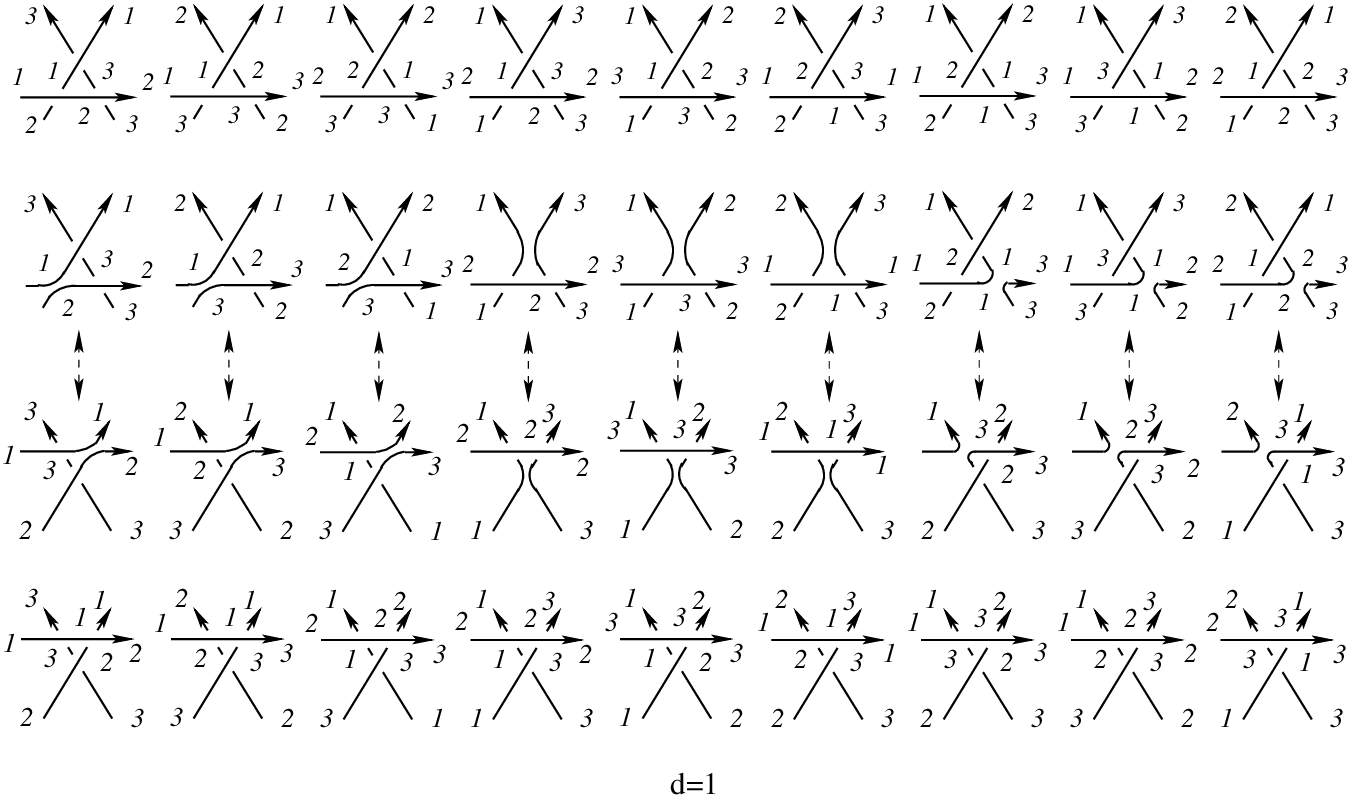}}
\caption{\label{fig:Rad3-3col1-inv} Correspondence of colorings with 1 special crossing and exactly 3 different colors in the fragments.}
\end{figure}
%%%%%%%%%%%%%%%%%%%%%%%%%%%%%%%%%%%%%%%%%%%
\begin{figure}[htb]
\centerline{\includegraphics[height=3.5in]{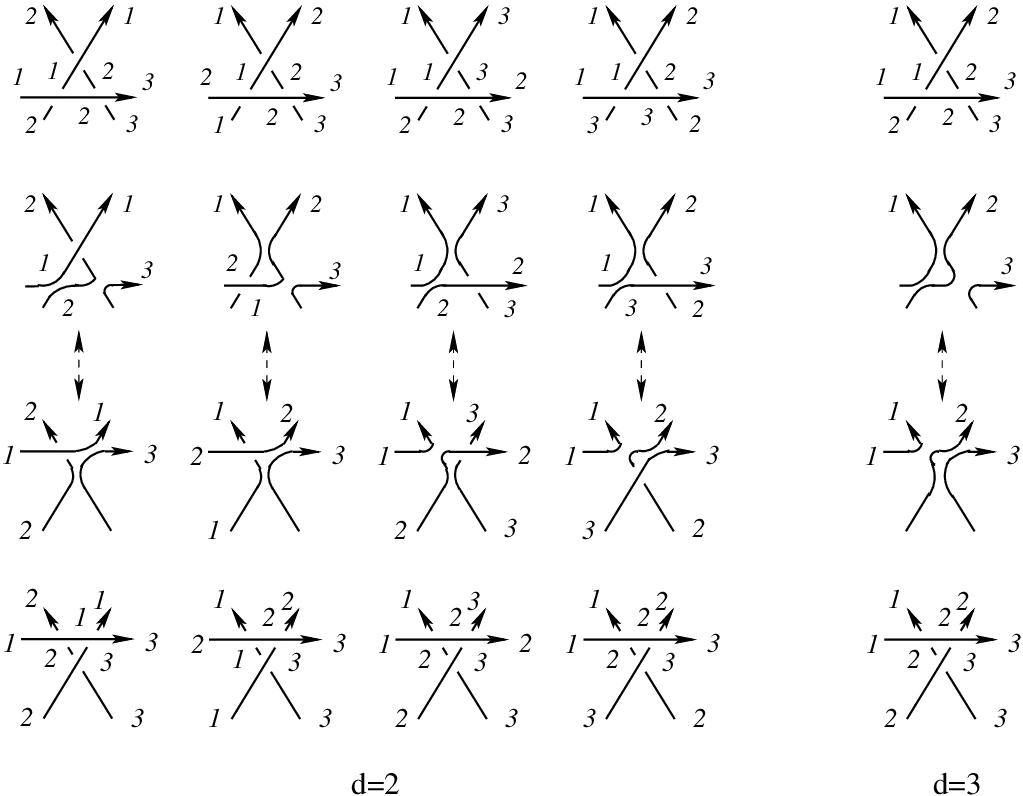}}
\caption{\label{fig:Rad3-3col23-inv} Correspondence of colorings with 2 and 3 special crossings and exactly 3 different colors in the fragments.}
\end{figure}
%%%%%%%%%%%%%%%%%%%%%%%%%%%%%%%%%%%%%%%%%%%

\textbf{Case 1.} Let $d=0$. There is a bijection between the sets $\C(D)_{j,0,3}$ and $\C(\wD)_{j,0,3}$. This bijection is shown in Figure \ref{fig:Rad3-3col0-inv}. It follows that for each pair of corresponding colorings $C$ and $\wC$ we have
$$I_{3,j}(D)_C=I_{3,j}(\wD)_{\wC},$$
and hence
\begin{equation}\label{eq:Rad3-3col0}
\sum_{C\in\C(D)_{j,0,3}}I_{3,j}(D)_C=\sum_{\wC\in\C(\wD)_{j,0,3}}I_{3,j}(\wD)_{\wC}.
\end{equation}
Same proof as the proof of case 1 in Proposition \ref{prop:O3} shows that
\begin{equation}\label{eq:Rad3-12col0}
\sum_{C\in\C(D)_{j,0}\setminus\C(D)_{j,0,3}}I_{3,j}(D)_C=\sum_{\wC\in\C(\wD)_{j,0}\setminus\C(\wD)_{j,0,3}}I_{3,j}(\wD)_{\wC}.
\end{equation}
\textbf{Case 2.} Let $1\leq d\leq 3$. In this case there is a correspondence between the sets $\C(D)_{j,d,3}$ and $\C(\wD)_{j,d,3}$. This bijection is shown in Figures \ref{fig:Rad3-3col1-inv} and \ref{fig:Rad3-3col23-inv}. It follows that for each pair of corresponding colorings $C$ and $\wC$ we have
$$I_{3,n,j}(D)_C=I_{3,n,j}(\wD)_{\wC},$$
and hence for each $1\leq d\leq 3$ we have
\begin{equation}\label{eq:Rad3-3col123}
\sum_{C\in\C(D)_{j,d,3}}I_{3,j}(D)_C=\sum_{\wC\in\C(\wD)_{j,d,3}}I_{3,j}(\wD)_{\wC}.
\end{equation}

Note that if $d=3$ then $\C(D)_{j,d,3}=\C(D)_{j,d}$ and $\C(\wD)_{j,d,3}=\C(\wD)_{j,d}$. It follows that
$\bigcup\limits_{d=1}^3(\C(D)_{j,d}\setminus \C(D)_{j,d,3})$ and $\bigcup\limits_{d=1}^3(\C(\wD)_{j,d}\setminus \C(\wD)_{j,d,3})$ contain colorings with \emph{1 or 2 special crossings} and with \emph{exactly 2 colors} in the distinguished fragments. Now the same proof as the proof of Proposition \ref{prop:O3} shows that
\begin{equation}\label{eq:Rad3-12col123}
\sum_{j=1}^\infty\sum_{d=1}^3\sum\limits_{C\in(\C(D)_{j,d}\setminus \C(D)_{j,d,3})}I_{3,j}(D)_C=
\sum_{j=1}^\infty\sum_{d=1}^3\sum\limits_{\wC\in(\C(\wD)_{j,d}\setminus \C(\wD)_{j,d,3})}I_{3,j}(\wD)_{\wC}.
\end{equation}

We combine equations \eqref{eq:Rad3-3col0}, \eqref{eq:Rad3-12col0}, \eqref{eq:Rad3-3col123} and \eqref{eq:Rad3-12col123}. This gives us the proof in the case of $k=3$ colors. The proof of the general case of $k>3$ colors follows immediately, since in this case the distinguished fragments of $D$ and $\wD$ may be colored by at most 3 different colors. Hence the same proof as in the case of $k=3$ colors proves the general case.
\qed

\begin{rem}\rm
Let $k\geq 3$ and $D$, $D'$ be two diagrams of the $(k-2)$-component unlink shown in Figure \ref{fig:unlink-Rad2}a. Then $I_k(D)=0$ and $I_k(D')=-\frac{k!}{6}z^2$. Let $D$ and $D''$ be two diagrams of the $(k-2)$-component unlink shown in Figure \ref{fig:unlink-Rad2}b. Then $I_k(D)=0$ and $I_k(D'')=-\frac{k!}{6}z^2$. This shows that the polynomials $I_k(D)$ and $I_k(1,z)(D)$ are not invariant under $\O2c$ and $\O2d$ Reidemeister moves and hence are not link invariants. However, a modification of the polynomials $I_k$ for $k\geq 3$, so that the resulting polynomials are link invariants, may be deduced from \cite{B1}.
\end{rem}
%%%%%%%%%%%%%%%%%%%%%%%%%%%%%%%%%%%%%%%%%%
\begin{figure}[htb]
\centerline{\includegraphics[height=1.6in]{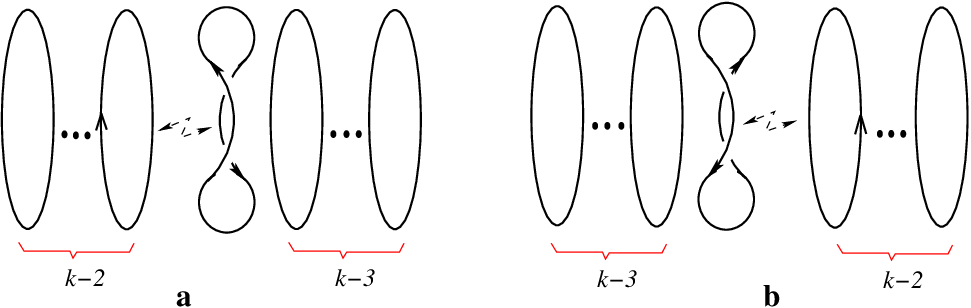}}
\caption{\label{fig:unlink-Rad2} Diagrams of the unknot which differ by $\O2c$ and $\O2d$ Reidemeister moves.}
\end{figure}
%%%%%%%%%%%%%%%%%%%%%%%%%%%%%%%%%%%%%%%%%%

\subsection{Proof of Theorem 3}

For each $j\geq 0$ and $d=0,1$ denote by $\C(D_+)_{j,d}$ and $\C(D_-)_{j,d}$ the subsets of $\C(D_+)_j$ and $\C(D_-)_j$ respectively which contain all colorings with $d$ \emph{special crossings} in the distinguished fragments of $D_+$ and $D_-$.\\

\textbf{Case 1.} Let $d=0$. For each color $1\leq p\leq k$ denote by $\C(D_+)_{j,0,p}$, $\C(D_-)_{j,0,p}$ and $\C(D_0)_{j,p}$ the subsets of $\C(D_+)_j$, $\C(D_-)_j$ and $\C(D_0)_j$ respectively which contain all colorings whose arcs in the distinguished fragments are colored by $p$, see Figure \ref{fig:D-triple-p}a.
%%%%%%%%%%%%%%%%%%%%%%%%%%%%%%%%%%%%%%%%%%
\begin{figure}[htb]
\centerline{\includegraphics[height=0.85in]{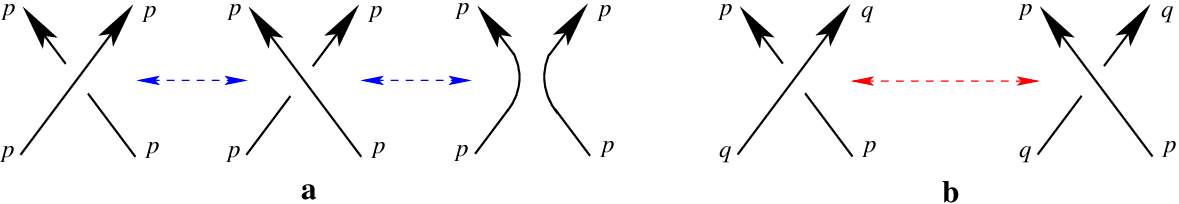}}
\caption{\label{fig:D-triple-p} Correspondence of colorings in case when $d=~0$. In Figure \textbf{b} we require that $p\neq q$.}
\end{figure}
%%%%%%%%%%%%%%%%%%%%%%%%%%%%%%%%%%%%%%%%%

Let $C_+$, $C_-$ and $C_0$ be the colorings in $\C(D_+)_{j,0,p}$, $\C(D_-)_{j,0,p}$ and $\C(D_0)_{j,p}$ respectively, such that they are identical outside the distinguished fragments. They induce link diagrams $D_{p,+}$, $D_{p,-}$ and $D_{p,0}$ which are colored by $p$ respectively. It follows from the HOMFLY-PT skein relation \eqref{eq:Homfly-skein} that
\begin{align*}
&a^{\o(D_{p,+})}P(L_{p,+})-a^{\o(D_{p,-})}P(L_{p,-})-za^{\o(D_{p,0})}P(L_{p,0})=\\
&a^{\o(D_{p,0})}\left(aP(L_{p,+})-aP(L_{p,-})-zP(L_{p,0})\right)=0
\end{align*}
Hence
$$I_{k,j}(D_+)_{C_+}-I_{k,j}(D_-)_{C_-}=zI_{k,j}(D_0)_{C_0}.$$
This yields
\begin{equation}\label{eq:skein-p}
\sum\limits_{C\in\C(D_+)_{j,0,p}}I_{k,j}(D_+)_{C}-\sum\limits_{C\in\C(D_-)_{j,0,p}}I_{k,j}(D_-)_{C}=
z\sum\limits_{C\in\C(D_0)_{j,p}}I_{k,j}(D_0)_{C}.
\end{equation}

There is a bijective correspondence between the sets $\C(D_+)_{j,0}\setminus\bigcup\limits_{p=1}^k\C(D_+)_{j,0,p}$ and $\C(D_-)_{j,0}\setminus\bigcup\limits_{p=1}^k\C(D_-)_{j,0,p}$. It is shown in Figure \ref{fig:D-triple-p}b. Note that for each two corresponding colorings $C_+$ and $C_-$ we have
$$I_{k,j}(D_+)_{C_+}=I_{k,j}(D_-)_{C_-}.$$
Combining this equality with \eqref{eq:skein-p} and summing over $p$ and $j$ we obtain
\begin{equation}\label{eq:skein-d=0}
\begin{array}{l}
\sum\limits_{j=0}^\infty\sum\limits_{C\in\C(D_+)_{j,0}}I_{k,j}(D_+)_{C}-
\sum\limits_{j=0}^\infty\sum\limits_{C\in\C(D_-)_{j,0}}I_{k,j}(D_-)_{C}= \\\\
z\sum\limits_{j=0}^\infty\sum\limits_{p=1}^k\sum\limits_{C\in\C(D_0)_{j,p}}I_{k,j}(D_0)_{C}.
\end{array}
\end{equation}
\textbf{Case 2.} Let $d=1$. Let $p,q\in\{1,\ldots,k\}$ such that $p<q$. Denote by $\C(D_+)_{j,(p,q)}$ and $\C(D_-)_{j,(q,p)}$ the subsets of $\C(D_+)_{j,1}$ and $\C(D_-)_{j,1}$ which contain all colorings such that the arcs in the distinguished fragment are colored as shown in Figures \ref{fig:D-triple-pq}a and \ref{fig:D-triple-pq}c respectively. We also denote by $\C(D_0)_{j,(p,q)}$ and $\C(D_0)_{j,(q,p)}$ the subsets of $\C(D_0)_j$ which contain all colorings such that the arcs in the distinguished fragment are colored as shown in Figures \ref{fig:D-triple-pq}b and \ref{fig:D-triple-pq}d respectively.
%%%%%%%%%%%%%%%%%%%%%%%%%%%%%%%%%%%%%%%%%
\begin{figure}[htb]
\centerline{\includegraphics[height=0.85in]{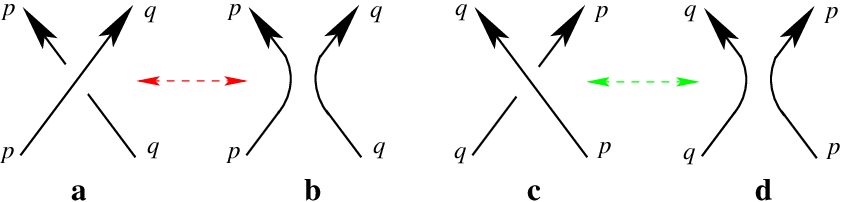}}
\caption{\label{fig:D-triple-pq} Correspondence of colorings in case when $d=1$.}
\end{figure}
%%%%%%%%%%%%%%%%%%%%%%%%%%%%%%%%%%%%%%%%%

There is a bijective correspondence between sets $\C(D_+)_{j,(p,q)}$ and $\C(D_0)_{j-1,(p,q)}$ as well as between sets $\C(D_-)_{j,(q,p)}$ and $\C(D_0)_{j-1,(q,p)}$. These bijections are shown in Figure \ref{fig:D-triple-pq}. They are presented by red and green arrows respectively. It follows that for each two corresponding colorings $C_+\in\C(D_+)_{j,(p,q)}$ and $C_0\in\C(D_0)_{j-1,(p,q)}$ we have
$$I_{k,j}(D_+)_{C_+}=zI_{k,j-1}(D_-)_{C_0},$$
and for each corresponding colorings $C_-\in\C(D_-)_{j,(q,p)}$ and $\widetilde{C_0}\in\C(D_0)_{j-1,(q,p)}$ we have
$$I_{k,j}(D_-)_{C_-}=-zI_{k,j-1}(D_-)_{\widetilde{C_0}}.$$
Summing over all such pairs $(p,q)$ and $j$ and noting that $\C(D_+)_{0,1}=\emptyset$ and $\C(D_-)_{0,1}=\emptyset$ we obtain
\begin{equation}\label{eq:skein-d=1}
\begin{array}{l}
\sum\limits_{j=0}^\infty\sum\limits_{C\in\C(D_+)_{j,1}}I_{k,j}(D_+)_{C}-
\sum\limits_{j=0}^\infty\sum\limits_{C\in\C(D_-)_{j,1}}I_{k,j}(D_-)_{C}=\\
z\sum\limits_{j=0}^\infty\sum\limits_{p<q}\left(\sum\limits_{C\in\C(D_0)_{j,(p,q)}}I_{k,j}(D_0)_{C}+
\sum\limits_{C\in\C(D_0)_{j,(q,p)}}I_{k,j}(D_0)_{C}\right).
\end{array}
\end{equation}
Now adding equations \eqref{eq:skein-d=0} and \eqref{eq:skein-d=1} we obtain
$$I_k(D_+)-I_k(D_-)=zI_k(D_0)$$
and the proof follows.
\qed

\subsection{Proof of Theorem 4}

By \eqref{I-2}, it is enough to show that the polynomials $I_2(1,z)(D)-w(D)z\n(L)$ and $zP'_a(L)|_{a=1}$ satisfy the same skein relation and receive the same values on the $r$-component unlink $O_r$. Let $D_r$ be a diagram of the unlink $O_r$, which consists of $r$ circles and no crossings. If $r\neq 2$, then $I_2(1,z)(D_r)=0$ because there are no colorings of $D_r$ with exactly 2 colors. If $r=2$, then there are exactly two colorings of $D_2$ with two colors. Hence by definition of $I_2(1,z)$ we have $I_2(1,z)(D_2)=2$. Note that for each $r$ we have $w(D_r)=0$. This yields
\begin{equation*}
I_2(1,z)(D_r)-w(D_r)z\n(O_r)=\begin{cases}
2,& \mbox{if}\quad r=2 \\
0, & \mbox{otherwise.}
\end{cases}
\end{equation*}
Let $D_+$, $D_-$ and $D_0$ be a Conway triple of link diagrams. It follows from Theorem \ref{thm:Conway-skein} and the Conway skein relation ~\eqref{eq:Conway-pol-skein} that
\begin{align*}
&I_2(1,z)(D_+)-w(D_+)z\n(L_+)-(I_2(1,z)(D_-)-w(D_-)z\n(L_-))+\\
&z\n(L_+)+z\n(L_-)=z(I_2(1,z)(D_0)-w(D_0)z\n(L_0)).
\end{align*}
The skein relation for the polynomial $zP'_a|_{a=1}$ follows directly from the skein relation for the HOMFLY-PT polynomial $P$. It is of the following form:
$$zP'_a(L_+)|_{a=1}-zP'_a(L_-)|_{a=1}+z\n(L_+)+z\n(L_-)=z^2P'_a(L_0)|_{a=1}.$$
We also have
\begin{equation*}
zP'_a(O_r)|_{a=1}=\begin{cases}
2,& \mbox{if}\quad r=2 \\
0,& \mbox{otherwise.}
\end{cases}
\end{equation*}
Hence both $I_2(1,z)(D)-w(D)z\n(L)$ and $zP'_a(L)|_{a=1}$ satisfy the same skein relation and normalization and the proof follows.
\qed

\subsection{Final questions and remarks}
\hfill
\begin{enumerate}
\item
Let $\Gamma$ be a group. Recall that
a function $\psi\colon \Gamma \to \B R$
is called a {\em quasi-morphism}
if there exists a real number $A\geq 0$ such that
$$
|\psi(gh) - \psi(g) - \psi(h)|\leq A
$$
Quasi-morphisms on groups of geometric origin
are related to different branches of mathematics see, for example \cite{Cal}. Let $J$ be a real-valued link invariant, then it defines a function
$$\widehat{J}\colon \B B_m\to \B R$$
by setting $\widehat{J}(\alpha):=J(\widehat{\alpha})$, where $\widehat{\alpha}$ is the link defined by a closure of the braid $\alpha$.

It is interesting to know whether some coefficients of the polynomial $I$, which was defined in Theorem 1, define quasi-morphisms, as explained above, on braid groups.

\item
Another interesting question is whether the polynomial $I$ is a complete invariant of conjugacy classes of braids. If yes, then it will give (plausibly the fastest) solution to the braid conjugacy problem.

\end{enumerate}

\bigskip

\textbf{Acknowledgments.} The author would like to thank Michael Polyak for helpful conversations. We would
like to thank the referee for careful reading of this paper and for his/her useful comments and remarks.

Part of this work has been done during the author's stay at Max Planck Institute for
Mathematics in Bonn. The author wishes to express his gratitude to the Institute for the support and excellent working conditions.

\bibliographystyle{alpha}

\begin{thebibliography}{99}
\addcontentsline{toc}{chapter}{\bibname}

\bibitem{BN1} Bar-Natan D.: \textit{On the Vassiliev knot invariants}, Topology, 34 (1995), 423-472.

\bibitem{Bir} Birman J.: \textit{New points of view in knot theory,} Bull. Amer. Math. Soc. 28 (1993), no. 2, 253-287.

\bibitem{B} Brandenbursky M.:
\textit{On quasi-morphisms from knot and braid invariants,} Journal of Knot Theory and Its Ramifications,
vol. 20, No 10 (2011), 1397--1417.

\bibitem{B1} Brandenbursky M.: \textit{Invariants of closed braids via counting surfaces,} Journal of Knot Theory and its Ramifications, vol. 22, No 3 (2013), 1350011 (21 pages).

\bibitem{B2} Brandenbursky M.:
\textit{Link invariants via counting surfaces,} arXiv:1209.0420, 2012.

\bibitem{BK} Brandenbursky M., Kedra J.:
\textit{On the autonomous metric on the group of area-preserving diffeomorphisms of the 2-disc}, Algebraic \& Geometric Topology,
\textbf{13} (2013), 795--816.

\bibitem{Cal} Calegari D.: \textit{scl}, MSJ Memoirs 20, Mathematical Society of Japan (2009).

\bibitem{CHH} Cochran T., Harvey S., Horn P.: \textit{Higher-order signature cocycles for subgroups of mapping class
groups and homology cylinders,} Int. Math. Res. Notices 2012 (14), 3311--3373.

\bibitem{Conway} Conway J.: \textit{An enumeration of knots and links,}
Computational problems in abstract algebra, Ed.J.Leech, Pergamon Press, (1969), 329-358.

\bibitem{FYHLMO} Freyd P., Yetter D., Hoste J., Lickorish W. B. R., Millett K., Ocneanu A.:
\textit{A new polynomial invariant of knots and links,} Bull. AMS 12 (1985), 239-246.

\bibitem{w-signature} Gambaudo J.-M., Ghys E.:
\textit{Braids and signatures,} Bull. Soc. Math. France 133, no. 4 (2005), 541--579.

\bibitem{surfaces} Gambaudo J.-M., Ghys E.:
\textit{Commutators and diffeomorphisms of surfaces,} Ergodic Theory Dynam. Systems 24, no. 5 (2004), 1591--1617.

\bibitem{GPV} Goussarov M., Polyak M., Viro O.:
\textit{Finite type invariants of classical and virtual knots,} Topology 39 (2000), 1045-1068.

\bibitem{HKM1} Honda K., Kazez W., Mati$\acute{\textrm{c}}$ G.: \textit{Right-veering diffeomorphisms of compact surfaces with boundary} I,
Invent. Math. 169 (2007), no. 2, 427--449.

\bibitem{HKM2} Honda K., Kazez W., Mati$\acute{\textrm{c}}$ G.: \textit{Right-veering diffeomorphisms of compact surfaces with boundary} II,
Geom. Topol. 12 (2008), no. 4, 2057--2094.

\bibitem{KT} Kassel C., Turaev V.: \textit{Braid groups,}
Graduate Texts in Mathematics 247, Springer, 2008.

\bibitem{Likorish} Lickorish W. B. R.:
\textit{An Introduction to Knot Theory,} 1997 Springer-Verlag New York, Inc.

\bibitem{LM} Lickorish W. B. R., Millett K.: \textit{A polynomial invariant of oriented links},
 Topology 26 (1) (1987), 107-141.

\bibitem{Mal2} Malyutin A.: \textit{Twist number of (closed) braids}, Algebra i Analiz 16 (2004), no. 5, 59-91; English
transl., St. Petersburg Math. J. 16 (2005), no. 5, 791--813.

\bibitem{Mal1} Malyutin A.: \textit{Pseudocharacters of braid groups and prime links,} Algebra i Analiz 21 (2009), no. 2,
113-135; English transl. in St. Petersburg Math. J. 21 (2010), no. 2.

\bibitem{P} Polyak M.: \textit{Minimal sets of Reidemeister moves,}
Quantum Topology 1 (2010), 399-411.

\bibitem{PV} Polyak M., Viro O.: \textit{Gauss diagram formulas for Vassiliev invariants},
Int. Math. Res. Notices 11 (1994), 445-454.

\bibitem{PT} Przytycki J., Traczyk P.: \textit{Invariants of links of the Conway type},
Kobe J. Math. 4 (1988), 115-139.

\bibitem{V2} Vassiliev V. A.: \textit{Complements of discriminants of smooth maps: topology and applications,} Trans. of Math.
Mono. 98, Amer. Math. Soc., Providence, 1992.

\bibitem{V1} Vassiliev V. A.: \textit{Cohomology of knot spaces,} Theory of Singularities and its Applications (Providence)
(V. I. Arnold, ed.), Amer. Math. Soc., Providence, 1990.




\end{thebibliography}

\bigskip

Max-Planck-Institut f$\ddot{\textrm{u}}$r Mathematik, 53111 Bonn, Germany\\
\emph{E-mail address:} \verb"brandem@mpim-bonn.mpg.de"

\end{document}